\documentclass[10pt]{article}
\usepackage{amssymb}
\usepackage{amsmath}
\usepackage{amsfonts}
\usepackage{amsthm}

\usepackage[english]{babel}
\usepackage[affil-it]{authblk}

\newtheorem{theorem}{Theorem}
\newtheorem{lemma}{Lemma}
\newtheorem{proposition}{Proposition}
\newtheorem{corollary}{Corollary}

\title{Equitable 2-partitions of the Hamming graphs with the second eigenvalue}

\author{Ivan Mogilnykh, Alexandr Valyuzhenich %
  \thanks{E-mail address: \texttt{ivmog@math.nsc.ru, graphkiper@mail.ru}}}
\affil{Sobolev Institute of Mathematics, pr. Akademika Koptyuga 4,
Novosibirsk 630090, Russia}
\begin{document}

\maketitle

\begin{abstract}
The eigenvalues of the Hamming graph $H(n,q)$ are known to be
$\lambda_i(n,q)=(q-1)n-qi$, $0\leq i \leq n$. The characterization
of equitable 2-partitions of the Hamming graphs $H(n,q)$ with
eigenvalue $\lambda_{1}(n,q)$ was obtained by Meyerowitz in
\cite{Mey}. We study the equitable 2-partitions of $H(n,q)$ with
eigenvalue $\lambda_{2}(n,q)$. We show that these partitions are
reduced to equitable 2-partitions of $H(3,q)$ with eigenvalue
$\lambda_{2}(3,q)$ with exception of two constructions.


\end{abstract}

\section{Introduction}
An ordered $r$-partition $(C_1,\ldots, C_{r})$ of the vertex set
of a graph is called {\it equitable} if  for any $i,j\in
\{1,\ldots,r\}$ there is $S_{ij}$ such that any vertex of $C_i$
has exactly $S_{ij}$ neighbors in $C_j$. The elements of the
partition are called {\it cells}. It is well-known (see e.g.
\cite{CDZ}) that the eigenvalues of the matrix $S=(S_{ij})_{i,j\in
\{1,\ldots,r\}}$, which is called the {\it quotient matrix} of the
equitable partition, are necessarily eigenvalues of the adjacency
matrix of the  graph. {\it An eigenvalue} of an equitable
partition is an eigenvalue of its quotient matrix.

Equitable partitions of various graphs are solutions for certain
covering and optimization problems in coding, graph and design
theories, projective geometry and etc. Such objects
 as
1-perfect codes, $(w-1)-(n,w,\lambda)$-designs \cite{Martin} and
their q-analogues, including $q$-ary Steiner triple systems
\cite{Bra}, spreads and Cameron-Liebler line classes in $PG(n,q)$
\cite{CL} can be defined in terms of equitable 2-partitions of
Hamming, Johnson and Grassman graphs.

A subset of the vertex set of the graph is called a {\it
completely regular code} if the distance partition with respect to
the subset is equitable. Obviously any cell of an equitable
2-partition is a completely regular code. Completely regular codes
in Hamming graphs include Preparata, some BCH and perfect codes.
For a survey on completely regular codes in Hamming graphs we
refer to \cite{ZinRifBorg}.

In throughout the paper we index the eigenvalues of the graphs in
the descending order, starting with zeroth. The eigenvalues of the
Hamming graph $H(n,q)$ are known to be $\lambda_i(n,q)=(q-1)n-qi$,
$i\in \{ 0,\ldots,n\}$, see e.g. \cite{BCN}.
 The characterization of all completely regular codes (in
particular equitable 2-partitions) with the first eigenvalue was
obtained for the Hamming and Johnson graphs by Meyerowitz in
\cite{Mey}. In the paper we consider the case when equitable
2-partitions of the Hamming graphs have the second eigenvalue.

In binary case the equitable 2-partitions were studied by
Fon-Der-Flaas \cite{FDF1,FDF2}. Consider an equitable 2-partition
of the Hamming graph $H(n,2)$ with an asymmetric quotient matrix
and eigenvalues $\lambda_i(n,2)$ and $n$. In \cite{FDF2}
Fon-Der-Flaas showed that the number $i$ of the minimum eigenvalue
is not greater than $2n/3$. Later in \cite{FDF3} he constructed
 2-partitions which eigenvalue numbers attain the above bound.
 Krotov and Vorob'ev \cite{KV} studied the existence problem of such
 partitions and obtained a new necessary condition in terms of elements of quotient matrices. They also gave a
 characterization of such partitions in $H(12,2)$ with only one open case left. The database of equitable 2-partitions of binary Hamming
 graphs can be found in \cite{KKM}.

 In case of arbitrary $q$, the equitable 2-partitions of $H(n,q)$ with the first
eigenvalue could be described as those that can be obtained from
2-partitions of $H(1,q)$ by adding $n-1$ nonessential coordinate
positions \cite{Mey}. The case of linear equitable 2-partitions,
i.e. such that the cells are cosets by a linear code of dimension
$n-1$ were characterized in \cite{BRZlin} (see also \cite{KLMT}).
Note that these partitions are related to Hamming codes and
therefore generally have numbers of their eigenvalues exceeding 2.

In the paper \cite{V1} every eigenvector of $H(n,q)$ with
eigenvalue $\lambda_i(n,q)$ was related to a set of eigenvectors
of $H(n-1,q)$ with eigenvalue $\lambda_{i-1}(n-1,q)$ in a certain
manner. Based on this connection the minimum weights of the
eigenvectors of $H(n,q)$ with eigenvalue $\lambda_1(n,q)$
\cite{V1} and arbitrary eigenvalue \cite{V2} were found. In case
of equitable partitions the approach of \cite{V1} relates every
equitable 2-partition of $H(n,q)$ with eigenvalue $\lambda_i(n,q)$
to a set of eigenvectors of $H(n-1,q)$ with eigenvalue
$\lambda_{i-1}(n-1,q)$ whose entries take values in the set
$\{0,1,-1\}$. We characterize all eigenvectors with eigenvalue
$\lambda_{1}(n-1,q)$ of the Hamming graph $H(n-1,q)$ taking values
$\{0,1,-1\}$ in Section \ref{sec:main}. The description of these
vectors impose restrictions on the structure of the parent
equitable 2-partition of $H(n,q)$ with eigenvalue
$\lambda_2(n,q)$, which are shown to be equitable 2-partitions of
$H(3,q)$ in Section \ref{sec:main} up to several constructions
that we give in Section \ref{sec:const}. The basic theory
concerning equitable partitions of the Hamming graphs and the
details of the approach from \cite{V1} are given in Section
\ref{sec:2}.

\section{Equitable partitions of the Hamming graphs}\label{sec:2}

Let $B_1,\ldots, B_n$ be finite sets, $x$ and $y$ be two tuples
from the cartesian product $B_1\times\ldots\times B_n$. We say
that $x$ is $s$-{\it adjacent} to $y$ if $x$ and $y$ differ only
in $s$th coordinate. The vertex set of the {\it Hamming graph}
$H(n,q)$ is the cartesian $n$th power of a set ${\cal A}$ of size
$q$ and vertices $x$ and $y$ are adjacent if they are differ in
exactly one coordinate position. 
Throughout the paper, ${\cal N}$ denotes the set $\{1,\ldots,
n\}$.

A function on the vertex set of a graph is called a $\lambda$-{\it
eigenfunction} if the  vector of its values is an eigenvector of
the adjacency matrix of the graph with eigenvalue $\lambda$ or the
all-zero vector.

Given a code (a set of vertices of a graph) $C$, by  $\chi_C$ we
denote the characteristic function of the code in the vertex set
of the graph.  Let $(u_1,\ldots,u_r)$ be an eigenvector of the
quotient matrix $S$ of an equitable partition $(C_1,\ldots,C_r)$
with eigenvalue $\lambda$. Then it is easy to see that the
function $\sum_{i\in \{1,\ldots,r\}} u_i\chi_{C_i}$ is a
$\lambda$-eigenfunction of the adjacency matrix of the graph,
which is known as Lloyd's theorem.

\begin{theorem}\cite{CDZ}
An eigenvalue of an equitable partition of a graph is an
eigenvalue of the graph.
\end{theorem}

Let $(C,\overline{C})$ be an equitable 2-partition of a
$k$-regular graph $G$ with the quotient matrix $S$. Then it is
easy to see that the eigenvalues of $S$ are $k$ and
$S_{11}-S_{21}$. Moreover, the considerations prior to  Lloyd's
theorem imply that $\chi_{C}$ is the sum of a $k$-eigenfunction
(which is a constant uniquely defined by $S$) and a two-valued
$(S_{11}-S_{21})$-eigenfunction of $G$ that is constant on $C$ and
$\overline{C}$. It is well-known that if the characteristic
function of a set in $H(n,q)$ is orthogonal to any
$\lambda_j(n,q)$-eigenfunction, for all $j$ such that $1\leq j\leq
t$, then the code is a $t-$orthogonal array (see e.g.
\cite{Del73}[Theorem 4.4]). In particular, in the case when
$(C,\overline{C})$ is an equitable 2-partition of $H(n,q)$ with
the second eigenvalue, the cell $C$ is evenly distributed in the
induced Hamming subgraphs $H(n-1,q)$.

\begin{proposition}\label{prop_orth_arr}
Let $(C,\overline{C})$ be an equitable 2-partition of a
$k$-regular graph $G$ with the quotient matrix $S$. Then the following statements hold:

1. $|C|=|V(G)|S_{21}/(S_{12}+S_{21})$ and the eigenvalues of the
partition are\\ $k, S_{11}-S_{21}$.

2. If $G$ is $H(n,q)$ and $S_{11}-S_{21}=\lambda_2(n,q)$, then for
any $i\in {\cal N}$, $\alpha\in {\cal A}$ we have that
$$|\{x\in C:x_i=\alpha\}|=S_{21}q^{n-2}/2.$$
\end{proposition}
\begin{proof}
Double counting of edges in $G$ between $C$ and $\overline{C}$
gives that $|C|=\frac{S_{21}}{S_{12}+S_{21}}|V(G)|$. When $G=H(n,q)$ and
$S_{11}-S_{21}=\lambda_2(n,q)=(q-1)n-2q$ we have
$S_{11}+S_{12}=n(q-1)$ which gives $|C|=S_{21}q^{n-1}/2$. Since
$C$ is a 1-orthogonal array, $|\{x\in C:x_i=\alpha\}|$ is a
constant number and we obtain the required.
\end{proof}


Let $f$ be a function defined on the vertices ${\cal A}^n$ of
$H(n,q)$, $\alpha \in {\cal A}$, $i\in {\cal N}$. Let
$(f)_i^{\alpha}$ be the function such that $(f)_i^{\alpha}(y')$ is
$f(y)$, where $y_i=\alpha$ and $y'$ is obtained from $y$ by
deleting $i$-th coordinate position. Denote by $U_i(n,q)$ the
space of $\lambda_i(n,q)$-eigenfunctions of $H(n,q)$.

\begin{lemma}\label{lem_smalred}\cite{V1}
 Let $f$ be a $\lambda_i(n,q)$-eigenfunction of $H(n,q)$. Then the following statements hold:

 1. For any $k \in {\cal N}$, $\alpha,\alpha' \in {\cal A}$ the function
$(f)_{k}^{\alpha}-(f)_k^{\alpha'}$ is
$\lambda_{i-1}(n-1,q)$-eigenfunction of $H(n-1,q)$.

2. For any $k\in {\cal N}$ and $\alpha\in {\cal A}$ the function
$(f)_{k}^{\alpha}$ is in $U_{i-1}(n-1,q)\oplus U_{i}(n-1,q)$.

\end{lemma}

\begin{corollary}\label{coro_1}
Let $(C,\overline{C})$ be an equitable 2-partition of $H(n,q)$
with eigenvalue $\lambda_i(n,q)$ and $k \in {\cal N}$, $\alpha,\beta \in {\cal A}$. Then the function
$(\chi_{C})_{k}^{\alpha}-(\chi_{C})_{k}^{\beta}$ is a
$\lambda_{i-1}(n-1,q)$-eigenfunction of $H(n-1,q)$ taking values
$\{0,1,-1\}$.
\end{corollary}
\begin{proof}
The function $\chi_{C}$ is the sum of a
$\lambda_i(n,q)$-eigenfunction and a constant function and the
latter vanishes in the expression
$(\chi_{C})_{k}^{\alpha}-(\chi_{C})_{k}^{\beta}$.
\end{proof}

\section{Constructions of equitable 2-partitions of H(n,q) with the second eigenvalue}
\label{sec:const}

Let $f$ be a function defined on the cartesian product
$B_1\times\ldots\times B_n$. We say $i\in {\cal N}$ is an {\it
essential coordinate} of $f$, if there are $i$-adjacent vertices
$x$ and $y$ such that $f(x)\neq f(y)$. A coordinate is {\it
essential} for a 2-partition of $B_1\times\ldots\times B_n$ if it
is essential for the characteristic function of one of its cells.
Given a function $f:{\cal{A}}^{n}\longrightarrow{\mathbb{R}}$,  we
define the function $f^{+}:{\cal
A}^{n+1}\longrightarrow{\mathbb{R}}$ by the rule
$f^{+}(x_1,\ldots,x_n,x_{n+1})=f(x_1,\ldots,x_n)$ for any
$x_{n+1}\in {\cal A}$.

The following result is a folklore, its proof is straightforward.
Lemma in case of equitable 2-partitions of the Hamming graph
$H(n,2)$ could be found in \cite{FDF1}.

\begin{lemma}\cite{FDF1}
\label{Lem_noness} 1. Let $f$ be a real-valued function defined on
${\cal A}^n$. Then $f$ is $\lambda_i(n,q)$-eigenfunction iff $f^+$
is $\lambda_i(n+1,q)$-eigenfunction.

2. Let $(C,\overline{C})$ be arbitrary 2-partition of $H(n,q)$,
$i$ be a nonessential coordinate of the partition. Then $(C,
\overline{C})$ is equitable with eigenvalue $\lambda_i(n,q)$ iff
the partition of $H(n-1,q)$ obtained by deleting $i$th position in
all tuples of $C$ and $\overline{C}$ is equitable with eigenvalue
$\lambda_i(n-1,q)$.
\end{lemma}

We say that a function defined on $B_1\times\ldots\times B_n$
(2-partition of $B_1\times\ldots\times B_n$ respectively) is {\it
reduced} if all coordinates from ${\cal N}$ are essential for the
function (the characteristic function of a cell of the partition
respectively).

\subsection{Permutation switching construction}

Here we present a construction of equitable 2-partitions of
$H(n,q)$ from specific equitable 2-partitions of $H(2,q)$ by
switchings of coordinate positions.

Recall that the {\em Cartesian product} $G\square H$ of graphs $G$ and $H$ is a graph with the vertex set $V(G)\times V(H)$; and
any two vertices $(u,u')$ and $(v,v')$ are adjacent if and only if either
$u=v$ and $u'$ is adjacent to $v'$ in $H$, or
$u'=v'$ and $u$ is adjacent to $v$ in $G$.

Firstly, we describe all equitable 2-partitions of
$H(1,q)\square H(1,q')$ with the smallest eigenvalue.

\begin{proposition}\label{Prop_ep_2}
A 2-partition $(C, \overline{C})$ of $H(1,q)\square H(1,q')$ is
equitable with the quotient matrix $S$ and eigenvalue $-2$ if and
only if $|C\cap K|/|K|=S_{21}/(q+q')$ for any maximal clique $K$
of $H(1,q)\square H(1,q')$.
\end{proposition}
\begin{proof}
Let $C$ be a subset of $H(1,q)\square H(1,q')$ such that any maximal
clique $K$ contains exactly $a|K|$ vertices of $C$. Then  the
partition $(C,\overline{C})$ is equitable with the quotient matrix $\left(%
\begin{array}{cc}
  a(q+q')-2 & (1-a)(q+q') \\
  a(q+q') & (1-a)(q+q')-2 \\
\end{array}%
\right)$ and has eigenvalue -2.

If $K$ is a maximal clique in $H(1,q)\square H(1,q')$, then it is
easy to see that the partition $(K,\overline{K})$ is an equitable
2-partition of $H(1,q)\square H(1,q')$ with the
quotient matrix $\left(%
\begin{array}{cc}
  |K|-1 & q+q'-1-|K| \\
  1 & q+q'-3 \\
\end{array}%
\right)$ that has eigenvalue $|K|-2$. Let $(C, \overline{C})$ be
an equitable partition of $H(1,q)\square H(1,q')$ with the
eigenvalue $-2$. Then $\chi_K=f_0+f_1$, $\chi_{C}=h_0+f_2$, where
$f_0$ and $h_0$ are constant functions that are uniquely defined
by the quotient matrices of the partitions, $f_1$ is a
$(|K|-2)$-eigenfunction and $f_2$ is a $(-2)$-eigenfunction. Since
eigenvectors of the adjacency matrix with different eigenvalues
are orthogonal,
 we have that:
$$ |C\cap K|=\sum_{v\in V(H(1,q)\square H(1,q'))}\chi_{K}(v)\chi_{C}(v)=\sum_{v\in V(H(1,q)\square H(1,q'))}
h_0(v)f_0(v),
$$
which does not depend on $K$, but only on $|K|$, since the value
of $f_0$ is uniquely determined by the quotient matrix of the
partition $(K,\overline{K})$. Since the graph $H(1,q)\square
H(1,q')$ could be parted into maximal cliques of the size $|K|$,
the number $|C\cap K|/|K|$ is equal to $|C|/qq'$. Finally, from
Proposition \ref{prop_orth_arr} we have
$|C|=S_{21}qq'/(S_{12}+S_{21})$, which taking into account that
$S_{12}+S_{21}=q+q'-2-S_{11}+S_{21}=q+q'$ implies that the
percentage of $C$ in $K$ is $S_{21}/(q+q')$.

\end{proof}

{\bf Construction A.}
 Let $A_1,\ldots,A_{n-1}$ be a partition of ${\cal A}$.
 The vertices of $H(2,q)$ are parted into $A_i\times {\cal A}$, $i=1,\ldots,n-1$.
 Consider a 2-partition $(C,\overline{C})$ of $H(2,q)$ such that
$|C\cap K|/|K|$ is the same for any maximal clique $K$ of any
subgraph induced by $A_i\times {\cal A}$, $i=1,\ldots,n-1$. By
Proposition \ref{Prop_ep_2} the restriction of $(C,\overline{C})$
to any of these subgraphs is equitable. Moreover,
$(C,\overline{C})$ is an equitable 2-partition of $H(2,q)$ with
the eigenvalue $\lambda_2(2,q)=-2$ as any maximal clique of
$H(2,q)$ is a union of maximal cliques of the subgraphs.

Define $C^{+}$ to be obtained from $C$ by adding $n-2$
nonessential coordinate positions, i.e.
$C^{+}=\{(x_1,x_2,y_3,\ldots,y_n): (x_1,x_2)\in C, y_i\in {\cal
A}, i \in \{3,\ldots,n\} \}.$

From Lemma \ref{Lem_noness} the partition
$(C^{+},{\overline{C^{+}}})$ of $H(n,q)$ is equitable with
eigenvalue $\lambda_2(n,q)$. The construction allows switchings of
coordinates to be applied. Given a permutation $\pi$ of the
coordinates from ${\cal N}$ and $x\in {\cal A}^n$ by $\pi(x)$
denote the tuple $(x_{\pi(1)},\ldots,x_{\pi(n)})$, for a subset
$M$ of ${\cal A}^n$ denote by $\pi(M)$ the set $\{\pi(x):x \in
M\}$. Let $\pi_i$ be the transposition $(2,i+1)$, and $\pi_1$ be
the identity permutation. Define $C_\pi^{+}$ to be
$$\bigcup_{i=1,\ldots,n-1} \pi_i( A_i\times {\cal A}^{n-1}\cap
C^+).$$ Note that the essential coordinates of
$(C_{\pi},\overline{C_{\pi}})$ are ${\cal N}$, while those of
$(C,\overline{C})$ are 1 and 2. The construction above is somewhat
similar to the construction of nonsystematic binary 1-perfect
codes by consecutive switchings of $i$-components \cite{AvgSol}.

\begin{theorem}
The partition $({C}^{+}_\pi, \overline{C^{+}_\pi})$ is an
equitable 2-partition of $H(n,q)$ with the same quotient matrix as
that of $(C^{+},\overline{C^{+}})$ and eigenvalue
$\lambda_2(n,q)$.
\end{theorem}
\begin{proof}
Let $S$ be the quotient matrix of the partition
$(C^{+},\overline{C^{+}})$. Define $(x)_{\pi}$ to be $\pi_i(x)$,
for all $i\in {\cal N}\setminus n$, $x\in A_i\times{\cal
A}^{n-1}$. The mapping $(\cdot)_{\pi}$ permutes the second and
$(i+1)$th coordinate positions for vertices in $A_i\times{\cal
A}^{n-1}$, so its restriction to the subgraph induced by
$A_i\times{\cal A}^{n-1}$ is an automorphism of the subgraph. The
permutation $\pi_i$ fixes 1, so the restriction of $(\cdot)_{\pi}$
to the subgraph induced by $A_i\times{\cal A}^{n-1}$ acts on the
set of maximal cliques consisting of 1-adjacent vertices of the
subgraph. By the choice of the initial partition
$(C,\overline{C})$, the percentages of $C^{+}$ in all maximal
cliques consisting of pairwise 1-adjacent vertices of the subgraph
are the same and are equal to $S_{21}/(S_{12}+S_{21})$. Since
$(C^{+})_{\pi}$ is $C_{\pi}$ and $(\cdot)_{\pi}$ permutes these
cliques,  the percentage of $C_{\pi}$ in these cliques is
$S_{21}/(S_{12}+S_{21})$.

For a vertex $x\in C^{+}$, $x_1\in A_i$ we see that
$(\cdot)_{\pi}$ maps the set of all $j$-neighbors from $C^{+}$,
for all $j\geq 2$, to the set of all $j$-neighbors of $(x)_{\pi}$
from $C_{\pi}$ because they are all in $A_i\times A^{n-1}$. A
maximum clique of $H(n,q)$ formed by pairwise 1-adjacent vertices
is the union of maximal cliques of the subgraphs induced by
$A_i\times{\cal A}^{n-1}$, $i \in {\cal N}\setminus n$. Since the
percentages of $C$ and $C_{\pi}$ in the cliques in subgraphs are
the same, we conclude that the numbers of
 vertices that are 1-adjacent to $x$ from $C^{+}$ and that of $(x)_{\pi}$
from $C_{\pi}$ coincide, so $(x)_{\pi}$ is adjacent to $S_{11}$
vertices of $C_{\pi}$. The proof in case when $x\in
\overline{C^{+}}$ is analogous.

\end{proof}

\subsection{Alphabet liftings of two induced cycles in H(4,2)}
Consider the construction of equitable partitions using alphabet
liftings by Vorob'ev in \cite{Vor}, which resembles the Zinoviev
construction for perfect codes \cite{Zin} (see also \cite{Sol}).
Let $A_0,\ldots,A_{q'-1}$ be the sets of the same size that
partition ${\cal A}$.  Here we identify the alphabet set ${\cal
A}'$ of $H(n,q')$ with $\{0,\ldots,q'-1\}$.

Given a partition $(C_1,\ldots,C_r)$ of $H(n,q')$ define $D_i$,
$1\leq i\leq r$ to be
$$\bigcup_{(x_1,\ldots,x_n)\in C_i}A_{x_1}\times\ldots\times A_{x_n}.$$

\begin{theorem}
\cite{Vor} Let $(C_1,\ldots,C_r)$ be an equitable partition of
$H(n,q')$ with the set of eigenvalues $\{\lambda_i(n,q'):i \in
I\}$. Then the partition $(D_1,\ldots,D_r)$ is an equitable
partition of $H(n,q)$ with the set of eigenvalues
$\{\lambda_i(n,q):i \in I\}$.
\end{theorem}

 Let
$C=\{(0001),(0011),(0010),(0110),(1110),(1100), (1101),(1001)\}$.
The complement $\overline{C}$ of $C$ is $\{(0000),(0100),(0101),
(0111),(1111),(1011),$ $(1010),(1000)\}$. We see that both $C$ and
$\overline{C}$ are induced cycles in $H(4,2)$, so any vertex of
$C$ is adjacent to exactly two vertices of $C$ and the partition
$(C,\overline{C})$
is equitable with the quotient matrix $\left(%
\begin{array}{cc}
  2 & 2 \\
  2 & 2 \\
\end{array}%
\right)$. Moreover, all 4 coordinates of the partition are
essential and $\lambda_2(4,2)=0$ is its eigenvalue.

\begin{corollary} (construction B)
Let $(C,\overline{C})$ be a partition of $H(4,2)$ into two induced
cycles of length 8. $A_0$, $A_1=\overline{A_0}$, $|A_0|=|A_1|=q/2$
and $D$ be $\bigcup_{(x_1,x_2,x_3,x_4)\in C}A_{x_1}\times
A_{x_2}\times A_{x_3}\times A_{x_4}$. Then $(D,\overline{D})$ is
an equitable partition of $H(4,q)$ with eigenvalue
$\lambda_2(4,q)$.
\end{corollary}

\section{Main results}\label{sec:main}
Here we obtain the description of $\lambda_1(n,q)$-eigenfunctions
with the values in $\{0,1,-1\}$ and utilize it for the
reconstruction of equitable 2-partitions of $H(n,q)$ with
eigenvalue $\lambda_2(n,q)$.

\subsection{$\lambda_1(n,q)$-eigenfunctions with the values in \{0,1,-1\}}

 Consider two functions on the
vertices of Hamming graphs. Given $i \in {\cal N}$, $A, B\subset
{\cal A}$, $A\cap B=\varnothing$, $A\cup B\neq \varnothing$  we
say that $f$ defined on ${\cal A}^n$ is the $(A,B,i)$-{\it quasi
string} if

$$f(x)=
\begin{cases}
    1, & x_i \in A \\
    -1, & x_i \in B \\
    0, & \hbox{otherwise}
\end{cases}.$$

Given $i,j \in {\cal N}$, $i\neq j$, and two nonempty subsets $A,
B$ of  ${\cal A}$ we say that $f$ defined on ${\cal A}^n$ is the
$(A,B,i,j)$-{\it quasi cross} if
$$f(x)=
\begin{cases}
    1, & x_i \in A, x_j \notin B\\
    -1, & x_i \notin A, x_j \in B  \\
    0, & \hbox{otherwise}
\end{cases}.$$

\begin{lemma}
Let $f$ be a function from ${\cal A}^n$ to $\{-1,0,1\}$. The
function $f$ belongs to ${U_{0}(n,q)\oplus U_{1}(n,q)}$ iff it is
the $(A,B,i,j)$-quasi cross or the $(A,B,i)$-quasi string or a
constant.
\end{lemma}
\begin{proof}
Let us prove this lemma by induction on $n$. If $n=1$, then it is
easy to see that $f$ is either a constant or the $(A,B,i)$-quasi
string.

Let us prove the induction step. Fix  $\beta \in {\cal A}$. Lemma
\ref{lem_smalred} implies that
$f_{n}^{\alpha}-f_{n}^{\beta}\in{U_{0}(n,q)}$ for any $\alpha\in
{\cal A}\setminus \beta$. Hence $f_{n}^{\alpha}-f_{n}^{\beta}$ is
a constant, which we denote
 by $c_{\alpha}$, so we have that
\begin{equation}\label{constant}
 f_{n}^{\alpha}\equiv f_{n}^{\beta}+c_{\alpha}.
\end{equation}

Now, we consider three cases. Given a function $g$ by $E(g)$ we
denote the set of its values.

\textbf{Case $1$.} In this case we suppose that
$\{-1,1\}\subseteq{E(f_{n}^{\beta})}$. Using the equality
(\ref{constant}), we see that $E(f_{n}^{\alpha})$ contains numbers
$c_{\alpha}-1$ and $c_{\alpha}+1$. Since
$E(f)\subseteq\{-1,0,1\}$, we have $c_{\alpha}=0$ for any
$\alpha\in{\mathcal{A}\setminus\{\beta\}}$. So
$f_{n}^{\alpha}\equiv f_{n}^{\beta}$ for any
$\alpha\in{\mathcal{A}\setminus\{\beta\}}$. By Lemma
\ref{lem_smalred} we obtain $f_{n}^{\beta}\in{U_{0}(n-1,q)\oplus
U_{1}(n-1,q)}$. Then using the induction assumption for
$f_{n}^{\beta}$, we finish the proof in this case.

\textbf{Case $2$.} Let $E(f_{n}^{\beta})$ be $\{0,1\}$ or
$\{-1,0\}$. Without loss of generality, we assume that
$E(f_{n}^{\beta})=\{0,1\}$. Using (\ref{constant}) and the fact
that $E(f)\subseteq\{-1,0,1\}$, we obtain that
$c_{\alpha}\in\{-1,0\}$ for any
$\alpha\in{\mathcal{A}\setminus\{\beta\}}$. Denote
$C=\{\alpha\in{\mathcal{A}\setminus\{\beta\}}:c_{\alpha}=-1\}$.

 By Lemma \ref{lem_smalred} we obtain $f_{n}^{\beta}\in{U_{0}(n-1,q)\oplus U_{1}(n-1,q)}$. Using the induction assumption for $f_{n}^{\beta}$, we obtain that $f_{n}^{\beta}$ is the $(A,B,i)$-quasi string. Moreover, $B=\emptyset$ because $E(f_{n}^{\beta})=\{0,1\}$. Then $f$ is the $(A,C,i,n)$-quasi cross due to the definition of the cross.

\textbf{Case $3$.} Let $f_{n}^{\beta}$ be a constant. Using
(\ref{constant}), we obtain that $f_{n}^{\alpha}$ is a constant
for any $\alpha\in{\mathcal{A}\setminus\{\beta\}}$. Then either
$f$ is a constant, or $f$ is the $(A,B,n)$-quasi string.

Let us now prove the sufficiency part. Obviously, the
$(A,B,1)$-quasi string belongs to $U_0(1,q)\oplus U_{1}(1,q)$
because the latter is the space of all real-valued functions on
$H(1,q)$. The $(A,B,i)$-quasi string in $H(n,q)$ is obtained from
the $(A,B,1)$-quasi string in $H(1,q)$ by adding $n-1$
nonessential coordinate positions, so by Lemma \ref{Lem_noness} it
is in $U_0(n,q)\oplus U_1(n,q)$.

 On the other hand, it is easy to see that the
$(A,B,i,j)$-quasi cross is the sum of the
$(A,\overline{A},i)$-quasi string and the
$(\overline{B},B,j)$-quasi string that belong to $U_0(n,q)\oplus
U_1(n,q)$, so the $(A,B,i,j)$-quasi cross is also in
$U_0(n,q)\oplus U_1(n,q)$.
\end{proof}

 It is easy to see that the $(A,B,i)$-quasi string is in
 $U_1(n,q)$ iff $|A|=|B|$. If $x$ is a zero of the $(A,B,i,j)$-quasi cross then $x$ is $i$-adjacent to $|A|$ vertices
$y$ such that $f(y)=1$ and $|A|$ vertices $y$ such that $f(y)=-1$.
If $x$ is such that $f(x)$ is positive (or negative) then $x$ is
$i$-adjacent or $j$-adjacent to $|A|+q-|B|-2$ vertices $y$ that
have positive (or negative) value $|B|+q-|A|-2$. Again, taking
into account Lemma \ref{Lem_noness}, the $(A,B,i,j)$-quasi cross
is $\lambda_1(n,q)$-eigenfunction iff $|A|=|B|$. When $|A|=|B|$
the $(A,B,i)$-quasi string is called the the $(A,B,i)$-{\it
string} and the $(A,B,i,j)$-quasi cross is called the
$(A,B,i,j)$-{\it cross}. The previous lemma implies  the
following.

\begin{lemma}\label{Lem_charl1}
Let $f$ be an arbitrary function from ${\cal A}^n$ to
$\{0,-1,1\}$. The function $f$ is $\lambda_1(n,q)$-eigenfunction
iff it is the $(A,B,i,j)$-cross or the $(A,B,i)$-string or the
all-zero function.
\end{lemma}

\subsection{Equitable 2-partitions of H(n,q) with eigenvalue $\lambda_2(n,q)$}

\begin{theorem}\label{the_1}
 Let $(C,\overline{C})$ be an equitable partition
of $H(n,q)$ with the quotient matrix $S$ that has eigenvalue
$\lambda_2(n,q)$. If $(\chi_C)_{k}^{\alpha'}-(\chi_C)_k^{\alpha}$
is the
$(A,B,i,j)$-cross 
and there is $s\in {\cal N}\setminus \{i,j,k\}$ that is essential
for $(\chi_C)_{k}^{\alpha}$ or $(\chi_C)_{k}^{\alpha'}$ then
$(C,\overline{C})$ is obtained by Construction B.
\end{theorem}
\begin{proof}
Without restriction of generality, suppose that $k=n$. Moreover,
 we have $|A|=|B|$, so up to an isomorphism of the Hamming graph we assume that
 $A=B$ and
$(\chi_C)_{k}^{\alpha'}-(\chi_C)_k^{\alpha}$ is the
$(A,A,i,j)$-cross. Then we have following the properties.

\hspace{-115mm}\begin{equation}\label{eq_1}
 \{x:x_n=\alpha', x_i\in A, x_j \in \overline{A}\},
\{x:x_n=\alpha, x_i\in \overline{A}, x_j \in A\}\subset
C,\end{equation}

\begin{equation}\label{eq_11}
\{x:x_n=\alpha, x_i\in A, x_j \in \overline{A}\}, \{x:x_n=\alpha',
x_i\in \overline{A}, x_j \in A\}\subset \overline{C}.
\end{equation}

\begin{gather}
 \mbox{Let } x \mbox{ and } y  \mbox{ be }n\mbox{-adjacent vertices } \mbox{ such that }  x_n=\alpha, y_n=\alpha',\nonumber\\
 \noindent x_i=y_i\in A,x_j=y_j\in A.  \mbox{ Then x and y are both in C or
not.}\label{eq_2}
\end{gather}

\begin{gather}
 \mbox{Let } x \mbox{ and } y  \mbox{ be }n\mbox{-adjacent
vertices } \mbox{ such that } x_n=\alpha, y_n=\alpha',\nonumber\\
 \noindent x_i=y_i\in \overline{A}, x_j=y_j\in \overline{A}. \mbox{ Then x
and y are both in C or not.}\label{eq_21}
\end{gather}

By conditions of the theorem there is a pair of $s$-adjacent
vertices $x^{0000}\in C$ and $x^{0001}\in\overline{C}$ such that
their $n$th positions are both $\alpha'$ or $\alpha$ and their
$i$th and $j$th positions are in
 $A$ or not simultaneously by properties (\ref{eq_1}), (\ref{eq_11}).
By properties (\ref{eq_2}) and (\ref{eq_21}) the essential
coordinate positions for $(\chi_C)_n^{\alpha}$ and
$(\chi_C)_n^{\alpha'}$ coincide, so we can assume that
$x^{0000}_n=x^{0001}_n=\alpha'$.

Moreover, we assume that both
$x^{0000}_i=x^{0001}_i=\beta',x^{0000}_j=x^{0001}_j=\gamma'$ are
in $\overline{A}$ (the case when they are in $A$ is proven
analogously).
  W.l.o.g. up to an automorphism of $H(n,q)$ we assume that
$\alpha$ and $x^{0001}_s=\delta\in A$,
 $\alpha'$ and $x^{0000}_s=\delta'\in \overline{A}$. By $\beta$ and $\gamma$ denote any two
 elements of $A$, so we have that $\alpha,\beta,\gamma,\delta \in A$,
$\alpha',\beta',\gamma',\delta' \in \overline{A}$.

 Denote by $x^{a_1
a_2a_3a_4}$ the vertex obtained from $x^{0000}$ by changing its
$n$th position from $\alpha'$ to $\alpha$ iff $a_1=1$, its $i$th
position from $\beta'$ to $\beta$ iff $a_2=1$, its $j$th position
from $\gamma'$ to $\gamma$ iff $a_3=1$, its $s$th position from
$\delta'$ to $\delta$ iff $a_4=1$. The graph spanned by $\{x^a:a
\in \{0,1\}^4\}$ is $H(4,2)$. We now show that the partition
$(C,\overline{C})$ could be reconstructed on the subgraph.

\begin{lemma}\label{lem_h42}
 $x^a\in C$ iff $a\in \{0000,0100,0101,1000,1010,1011,0111,1111\}$
\end{lemma}

\begin{proof}

Note that by (\ref{eq_2}) and (\ref{eq_21}) for any $a\in
\{0,1\}^3$ the vertex $x^{0a}$ is in $C$ iff $x^{1a}$ is in $C$
when $x_i^{0a}, x_j^{0a}\in A$ or $x_i^{0a},x_j^{0a}\in
\overline{A}$. Using this and taking into account (\ref{eq_1}) and
(\ref{eq_11}) we have that
\begin{equation}\label{eq_C}x^a\in C \mbox{ if } a\in
\{0000,0100,0101,1000,1010,1011\},\end{equation}
\begin{equation}\label{eq_C1}x^a\in \overline{C} \mbox{ if }
a\in\{0001,0010,0011,1001,1100,1101\}.\end{equation}

We now show that $x^{0110}$ is in $\overline{C}$ and $x^{0111}$ is
in $C$. Denote by $x^{a_1a_2a_4}$ the vertex obtained from
$x^{a_1a_2a_3a_4}$ by deleting its $j$th coordinate position.

Consider the values of the function
$(\chi_C)_{j}^{\gamma'}-(\chi_C)_{j}^{\gamma}$ on
$\{x^{a_1a_2a_3}:a_l\in\{0,1\},l=1,2,3\}$, $\gamma'=x_j^{0000}$,
$\gamma=x_j^{0010}$. From (\ref{eq_C}) and (\ref{eq_C1}) we have
that
$$((\chi_C)_{j}^{\gamma'}-(\chi_C)_j^{\gamma})(x^{000})=\chi_C(x^{0000})-\chi_C(x^{0010})=1,$$
$$((\chi_C)_{j}^{\gamma'}-(\chi_C)_j^{\gamma})(x^{001})=\chi_C(x^{0001})-\chi_C(x^{0011})=0,$$
$$((\chi_C)_{j}^{\gamma'}-(\chi_C)_j^{\gamma})(x^{100})=\chi_C(x^{1000})-\chi_C(x^{1010})=0.$$

If $x^{0110}$ is in $C$ then
$$((\chi_C)_{j}^{\gamma'}-(\chi_C)_j^{\gamma})(x^{010})=\chi_C(x^{0100})-\chi_C(x^{0110})=0.$$
If $x^{0111}$ is not in $C$ then
$$((\chi_C)_{j}^{\gamma'}-(\chi_C)_j^{\gamma})(x^{011})=\chi_C(x^{0101})-\chi_C(x^{0111})=1.$$

Therefore, if $x^{0110}$ is in $C$ or $x^{0111}$ is in
$\overline{C}$ then $(\chi_C)_{j}^{\gamma'}-(\chi_C)_{j}^{\gamma}$
has at least three essential coordinate positions. However, by
Corollary \ref{coro_1}
$(\chi_C)_{j}^{\gamma'}-(\chi_C)_{j}^{\gamma}$ is a
$\lambda_1(n-1,q)$-eigenfunction of $H(n-1,q)$, so it has not more
then two essential coordinates according to the characterization
in Lemma \ref{Lem_charl1}. Therefore we have that $x^{0110}\in
\overline{C}$ and $x^{0111}\in C$ and then $x^{1110}\in
\overline{C}$ and $x^{1111}\in C$ since changing $n$th position to
$\alpha$ preserves the property of being in $C$ in this case by
properties (\ref{eq_2}) and (\ref{eq_21}).

\end{proof}

\begin{lemma}
We have that $S_{12}=S_{21}=q$, $|A|=q/2$ and the following holds
up to an isomorphism of $H(n,q)$.

\begin{equation}\label{eq_4}\mbox{ A vertex }x\mbox{ such that }x_i,x_j\in A \mbox{
is in }C \mbox{ iff } x_s\in \overline{A},
\end{equation}
\begin{equation}
\label{eq_41}\mbox{ A vertex }x\mbox{ such that }x_i,x_j\in
\overline{A}\mbox{ is in }C \mbox{ iff } x_s\in A. \end{equation}

\end{lemma}

\begin{proof}

Let $A'$ be $\{x'_s:x'\in \overline{C} \mbox{ is s-adjacent to
}x^{0000}\}$. Lemma \ref{lem_h42} holds for any vertex
$x^{0001}\in \overline{C}$ that is s-adjacent to $x$, so we have
that
\begin{equation}\label{eq_0}{\cal A}\setminus A'=\{x_s:x\in \overline{C} \mbox{ is s-adjacent to
}x^{0111}\}.\end{equation}

We now evaluate the numbers of the neighbors of $x^{0000}$,
$x^{0111}$ from $\overline{C}$. By the definition of equitable
partition, it is $S_{12}$, where $(S_{ij})_{i,j=1,2}$ is the
quotient matrix of the partition $(C,\overline{C})$. Consider the
vertices obtained from $x^{0000}\in C$ by changing its $j$th
symbol to an element from $A$.
 By property (\ref{eq_11}) these vertices are not in $C$, so there are at least
$|A|+|A'|$ vertices from $\overline{C}$, that are $s$- or
$j$-adjacent to $x^{0000}$. By property (\ref{eq_11}) the vertices
obtained from $x^{0111}\in C$ by changing its $i$th symbol to an
element of $\overline{A}$ are in $\overline{C}$. The above
combined with (\ref{eq_0}) gives that

\begin{equation}\label{eq_3}S_{12}\geq max\{|A|+|A'|, 2q-|A'|-|A|\}.\end{equation}

Now count the neighbors of $x^{0001}$ and $x^{0110}$ in $C$. For
$x^{0001}$ there are $|A|$ $i$-neighbors and $q-|A'|$
$s$-neighbors from $C$. For $x^{0110}$ there are $q-|A|$
$j$-neighbors and $|A'|$ $s$-neighbors from $C$. Then we have that
\begin{equation}\label{eq_31}S_{21}\geq max\{q+|A|-|A'|,q-|A|+|A'|\}.\end{equation}

Since $\lambda_2(n,q)=n(q-1)-2q=S_{11}-S_{21}$ is an eigenvalue of
the equitable partition with the quotient matrix $S$ we have that
$S_{12}+S_{21}=2q$, which combined with the inequalities
(\ref{eq_3}) and (\ref{eq_31}) implies that $|A'|=|A|=q/2$ (in
below up to an automorphism of $H(n,q)$ we assume that $A'=A$) and
$S_{12}=S_{21}=q$. Moreover, the bound (\ref{eq_3}) is attained.
This implies that the neighbors of $x^{0000}$ from $\overline{C}$
as exactly those $j$- and $s$-neighbors of $x^{0000}$ counted
while obtaining (\ref{eq_3}). So all $l$-adjacent vertices of
$x^{0000}$, $l\in {\cal N}\setminus \{j,s\}$ are in $C$.
Analogously, the vertices from $C$ adjacent to the vertex
$x^{0001}$ are exactly those $i$-and $s$-adjacent vertices to
$x^{0001}$ counted while obtaining bound (\ref{eq_31}). We
conclude that all vertices that are $l$-adjacent vertices to
$x^{0001}$, $l\in {\cal N}\setminus \{i,s\}$ are in
$\overline{C}$. The consideration above holds for any two
$s$-adjacent vertices $x^{0000}\in C$ and $x^{0001}\in
\overline{C}$ and we know their neighbors from $\overline{C}$ and
$C$ respectively, so (\ref{eq_4}) and (\ref{eq_41}) hold.

\end{proof}

By the previous Lemma, (\ref{eq_1}) and (\ref{eq_2}) the partition
$(C,\overline{C})$ is reconstructed on the following set $\{x\in
{\cal A}^n:x_n=\alpha,\alpha'\}$. We now show that for any
$\beta\in {\cal A}$ the vertices of $\{x\in {\cal
A}^n:x_n=\beta\}\cap C$ are obtained by "copying" their
n-neighbors  either from $\{x:x_n=\alpha\}\cap C$ or from
$\{x:x_n=\alpha'\}\cap C$.

 Consider $f=(\chi_C)^{\alpha}_n-(\chi_C)^{\beta}_n$, $\beta\in
{\cal A}\setminus \{\alpha,\alpha'\}$. Then we have that
\begin{equation}\label{eq_l}
f(y)=0\mbox{ if } y_i, y_j\in A\mbox{ or } y_i, y_j\in
\overline{A}.
\end{equation}
We now show that $f$ (and therefore $(\chi_C)^{\beta}_n$) could be
reconstructed in only two ways.

From (\ref{eq_1}), we have that $f(y)\leq 0$ if $y_i\in A$,
$y_j\in \overline{A}$. We show that $f(y)$ is a constant for all
$y \in {\cal A}^{n-1}$, $y_i\in A$, $y_j\in \overline{A}$. Let $y$
be such that $f(y)=-1$, $y_i\in A, y_j \in \overline{A}$. Then
from (\ref{eq_l}) $i$ and $j$ are essential coordinates for $f$
and by Lemma \ref{Lem_charl1} we see that $f$ is the
$(A',B',i,j)$-cross. Moreover, if $z$ is $i$- or $j$-adjacent to
$y$, $f(z)=0$, $z_i\in A,z_j\in \overline{A}$ then by (\ref{eq_l})
a nonzero $y$ of the function $f$ is adjacent to at least $q+1$
zeros of $f$, which contradicts the definition of the
$(A',B',i,j)$-cross. Therefore we have that $A'=B'=A$ and from
(\ref{eq_1}) and (\ref{eq_2}) we reconstruct $C$ as follows:
\begin{equation}\label{rec_1}
\mbox{n-adjacent vertices of } \{x:x_n=\beta\}\mbox{ and }
\{x:x_n=\alpha'\} \mbox{ are both in }C \mbox{ or not}
\end{equation}

Let all $y \in {\cal A}^{n-1}$ such that $y_i\in A, y_j\in
\overline{A}$ be zeros of $f$. Then from (\ref{eq_l}) we have that
$f(y)=0$ when $y_i\in A$ or $y_j\in \overline{A}$ which implies
that $f$ is either the all-zero function or has two nonessential
coordinates by Lemma \ref{Lem_charl1}. However, in the latter case
a vertex $y$, $y_i\in A$, $y_j\in\overline{A}$ is adjacent only to
zeros of $f$, which contradicts the definition of the cross.
Therefore $f$ is the all-zero function and we reconstruct $C$ as
follows:
\begin{equation}\label{rec_2}
\mbox{n-adjacent vertices of } \{x:x_n=\beta\}\mbox{ and }
\{x:x_n=\alpha\} \mbox{ are both in }C \mbox{ or not}
\end{equation}

We now show that there are exactly $q/2$ elements $\beta\in {\cal
A}$ that satisfy (\ref{rec_1}). Let $z\in C$ be such that
$z_n=\alpha'$, $z_i\in A, z_j\in \overline{A}$. Then by
(\ref{eq_2}) the vertex $z$ is not $l$-adjacent to vertices of
$\overline{C}$, $l\in {\cal N}\setminus i,j,n$ and by (\ref{eq_4})
and (\ref{eq_41}) the vertex $z$ is $i$- or $j$-adjacent to
exactly $q/2$ vertices of $\overline{C}$. Therefore, $z$ is
$n$-adjacent with exactly $q/2=S_{12}-q/2$ vertices of
$\overline{C}$. In other words there are exactly $q/2$ elements
$\beta\in {\cal A}$ that satisfy (\ref{rec_1}).

 Finally, from (\ref{eq_4}),
(\ref{eq_41}), (\ref{rec_1}), (\ref{rec_2}) we see that  $x$ is in
$C$ iff
$$(x_n,x_i,x_j,x_s)\in A\times A\times A\times \overline{A}\cup
\overline{A}\times A\times A\times\overline{A}\cup
 A\times A\times\overline{A}\times\overline{A}\cup$$
$$ A\times A\times\overline{A}\times A\cup
 A\times \overline{A}\times\overline{A}\times A\cup
 \overline{A}\times \overline{A}\times\overline{A}\times A\cup$$
 $$ \overline{A}\times \overline{A}\times A\times \overline{A}\cup
\overline{A}\times \overline{A}\times A\times A, $$

i.e. the partition $(C,\overline{C})$ is obtained from
Construction $B$.

\end{proof}

The main result of this section is

\begin{theorem}
The only  reduced equitable
2-partitions of $H(n,q)$ with eigenvalue $\lambda_2(n,q)$ are: \\
\begin{enumerate}
    \item Reduced equitable 2-partitions of $H(2,q)$ and $H(3,q)$
    \item If $q$ is even, the equitable 2-partition of $H(4,q)$ from
alphabet liftings of two induced 8-cycles in $H(4,2)$
(Construction B).
    \item Equitable 2-partitions of $H(n,q)$ obtained by the permutation
switchings. (Construction A).
\end{enumerate}

\end{theorem}
\begin{proof}

Consider $(\chi_C)^{\alpha}_n$, $\alpha\in {\cal A}$. Define the
partition of ${\cal A}$ into sets $A_1,\ldots,A_t$ that we call
blocks as follows: $\alpha$, $\beta$ are in one block iff
$(\chi_C)^{\alpha}_n$ and $(\chi_C)^{\beta}_n$ have the same sets
of essential coordinate positions. The proof follows from Lemmas
\ref{lem_l8} and \ref{lem_l9} below.

\begin{lemma}\label{lem_twobl1}
If there are at least two blocks, then for any $\alpha\in {\cal
A}$ the function $(\chi_C)_{n}^{\alpha}$ has exactly one essential
coodinate.
\end{lemma}
\begin{proof}
Let $\alpha$ and $\beta$ be in different blocks. Then by Lemma
\ref{Lem_charl1} $(\chi_C)_n^{\alpha}-(\chi_C)_n^{\beta}$ is
either the $(A,B,i)$-string or the $(A,B,i,j)$-cross.

The function $(\chi_C)_n^{\alpha}-(\chi_C)_n^{\beta}$ is not
$(A,B,i)$-string because $\alpha$ and $\beta$ are in different
blocks. Suppose the opposite. Obviously, $i$ is an essential
coordinate for both $(\chi_C)_n^{\alpha}$ and
$(\chi_C)_n^{\beta}$. Moreover, $l\in {\cal N}\setminus \{i,n\}$
is not an essential coordinate for
$(\chi_C)_n^{\alpha}-(\chi_C)_n^{\beta}$, so for any pair of
$l$-adjacent tuples $y$ and $y'$ of ${\cal A}^{n-1}$, we have that
$(\chi_C)_n^{\alpha}(y)-(\chi_C)_n^{\beta}(y)=(\chi_C)_n^{\alpha}(y')-(\chi_C)_n^{\beta}(y')$.
In other words, we have that
$(\chi_C)_n^{\alpha}(y)-(\chi_C)_n^{\alpha}(y')=(\chi_C)_n^{\beta}(y)-(\chi_C)_n^{\beta}(y')$,
so $l$ is an essential coordinate for both $(\chi_C)_n^{\alpha}$
and $(\chi_C)_n^{\beta}$ or not. We conclude that
$(\chi_C)_n^{\alpha}$ and $(\chi_C)_n^{\beta}$ are in one block.

Let $(\chi_C)_n^{\alpha}-(\chi_C)_n^{\beta}$ be $(A,B,i,j)$-cross.
If $(\chi_C)_n^{\alpha}$ or $(\chi_C)_n^{\beta}$ has an essential
coordinate $s\in {\cal N}\setminus \{i,j,n\}$, then by
 Theorem \ref{the_1} we have Construction B and only one block in
 this case. So, we conclude that $(\chi_C)_n^{\alpha}$ and
 $(\chi_C)_n^{\beta}$ have exactly one essential coordinate in
 $\{i,j\}$.

\end{proof}

\begin{lemma}\label{lem_l8}
If there are at least two blocks then $(C,\overline{C})$ is
obtained by Construction A.
\end{lemma}
\begin{proof} By Lemma \ref{lem_twobl1} the number of blocks is
greater then the number of the essential coordinate positions of
the partition by 1. Since the partition is reduced, the number of
essential coordinates is $n$ and $A_1,\ldots,A_{n-1}$ are the
blocks of the partition $(C,\overline{C})$. W.l.o.g. for any
$\alpha\in A_i$ $i$ is the essential coordinate for
$(\chi_C)_n^{\alpha}$.

Consider a vertex whose $n$th position is in $A_i$. Taking into
account Proposition \ref{prop_orth_arr} and because each of
$(\chi_C)_{n}^{\alpha}$ has exactly one essential coordinate  by
Lemma \ref{lem_twobl1}, $|C\cap K_i|=S_{21}/2$, $|\overline{C}\cap
K_i|=S_{12}/2$ for the maximum clique $K_i$ consisting of the
vertex and its $i$-neighbors. Let the vertex be from $C$
($\overline{C}$ respectively). Then since $i$ is the only
essential coordinate for $(\chi_C)_n^{\alpha}$, $\alpha\in A_i$,
the vertex is not $j$-adjacent to any vertices of $\overline{C}$
($C$ respectively ) for $j \in {\cal N}\setminus \{i,n\}$ and
therefore is $n$-adjacent to exactly $S_{12}/2$ vertices of
$\overline{C}$ ($S_{21}/2$ vertices of $C$). So we see that any
maximum clique consisting of pairwise $n$-adjacent vertices of
$H(n,q)$ contains exactly $S_{12}/2$ vertices of $\overline{C}$
and $S_{21}/2$ vertices of $C$.

Let $K_n$ and $K'_n$ be two maximum cliques consisting of pairwise
$n$-adjacent vertices of $H(n,q)$. For $i\in {\cal N}\setminus n$
denote by $L_i$ ($L'_i$) those vertices of $K_n$ and ($K'_n$
respectively) that have their $n$th coordinate in $A_i$. From the
shown above we have that

\begin{equation}\label{eq_cl_subcli}|K_n\cap \overline{C}|=|K'_n\cap \overline{C}|=\sum_{i\in
{\cal N}\setminus n} |L_i\cap
\overline{C}|=S_{12}/2.\end{equation}

 We now prove that for any $i \in {\cal N}\setminus n$ we have
 that
$$|L_i\cap \overline{C}|=|L_i'\cap \overline{C}|=S_{12}|L_i|/2q.$$

 It is sufficient to prove the equality above when
for some $\beta\in {\cal A}$
 the vertices of $K'_n$ is obtained from the vertices of $K_n$ (and therefore $L_j$ and $L'_j$ for any $j\in {\cal N}\setminus \{i,n\}$) by
 changing their $i$th coordinate position to $\beta$. We see that $i$-adjacent pairs of vertices between $L_j$ and $L'_j$ is a
perfect matching. This, combined with the fact that $i$ is not
essential for the restriction of $(C,\overline{C})$ to ${\cal
A}^{n-1}\times A_j$ gives $|L_j\cap \overline{C}|=|L'_j\cap
\overline{C}|$. Then from (\ref{eq_cl_subcli}) we have that
$|L_i\cap \overline{C}|=|L'_i\cap \overline{C}|$.

For $j$ in ${\cal N}\setminus n$ consider the restriction of the
partition $(C,\overline{C})$ to the subgraph induced by ${\cal
A}^{n-1}\times A_j$. Each maximum clique of the subgraph
consisting of pairwise $j$-adjacent vertices contains $S_{12}/2$
vertices of $\overline{C}$ and any maximal clique of the subgraph
consisting of pairwise $n$-adjacent vertices contains has exactly
$l_j$ vertices of $\overline{C}$.  Since the coordinates from
${\cal N}\setminus \{j,n\}$ are nonessential for the restriction,
double counting of $\overline{C}$ in ${\cal A}^{n-1}\times A_j$
gives that $l_j/|A_j|=S_{12}/2q$ is a constant regardless of $j$.
We have shown that the partition $(C,\overline{C})$ is equally
distributed by all maximum cliques $K_j$, $j\in {\cal N}\setminus
n$, whose vertices have $n$th coordinate in $A_j$ and by the sets
$x'\times A_j$, for all $x'\in {\cal A}^{n-1}$, $j\in {\cal N}$.
We conclude that $(C,\overline{C})$ is obtained by Construction A
by the transposition of coordinate positions 1 and $n$.
\end{proof}

\begin{lemma}\label{lem_l9}
If there is one block then $(C,\overline{C})$ has not more then 3
essential coordinates or $(C,\overline{C})$ is obtained by
Construction $B$.
\end{lemma}
\begin{proof}
Let $\alpha$, $\beta$ be distinct elements of ${\cal A}$. Consider
$(\chi_C)_{n}^{\alpha}- (\chi_C)_{n}^{\beta}$. By Lemma
\ref{Lem_charl1} and Theorem \ref{the_1} we have the following
cases:
\begin{itemize}
    \item $(\chi_C)_{n}^{\alpha}- (\chi_C)_{n}^{\beta}$ is the
     $(A,B,i,j)$-cross and there is $s\in{\cal N}\setminus
    \{i,j,n\}$ that is essential for $(\chi_C)_{n}^{\alpha}$
    and $(\chi_C)_{n}^{\beta}$. Then $(C,\overline{C})$
     is obtained by Construction B.
     \item $(\chi_C)_{n}^{\alpha}- (\chi_C)_{n}^{\beta}$ is the
     $(A,B,i,j)$-cross, the essential coordinates
     of $(\chi_C)^\alpha_n$ and $(\chi_C)^\beta_n$ are in $\{i,j\}$. Since $\alpha$ and $\beta$ are from the same block, $i$ and $j$ are the essential coordinates
     for both $(\chi_C)^\alpha_n$ and $(\chi_C)^\beta_n$. The partition $(C,\overline{C})$ has exactly three
     essential coordinates in this case.
     \item $(\chi_C)_{n}^{\alpha}- (\chi_C)_{n}^{\beta}$ is the
     $(A,B,i)$-string.
\end{itemize}
Let $(\chi_C)_{n}^{\alpha}- (\chi_C)_{n}^{\beta}$ be the
$(A,B,i)$-string. From the definition of the string we have that
$x\in C$ if $x_n=\alpha,x_i\in A$ or $x_n=\beta,x_i\in B$ and
$x\in \overline{C}$ if $x_n=\beta,x_i\in A$ or $x_n=\alpha,x_i\in
B$.


Suppose that $s, s\neq i,n$ is an essential coordinate of
$(\chi_C)_{n}^{\alpha}$ and $(\chi_C)_{n}^{\beta}$ and there are
s-adjacent vertices for $x^{000}\in C$ and $x^{001}\in
\overline{C}$, $x_n^{000}=x_n^{001}=\alpha$,
$x_i^{000}=x_i^{001}=\tilde{\alpha}$. Fix $\alpha'\in A$ and
denote by $x^a$ the vertex which is obtained from $x^{000}$ by
changing its $n$th position to $\beta$ iff $a_1=1$ and $i$th
position to $\alpha'$ iff $a_2=1$ and $s$th position to
$x_s^{001}$ iff $a_3=1$. By the properties above we know that
$x^a\in C$ iff $a\in\{000,010,011,100\}$. Consider
$f=(\chi_C)_i^{\tilde{\alpha}}-(\chi_C)_i^{\alpha'}$ on the tuples
$\{x^{a_1a_3}:a_1,a_3\in \{0,1\}\}$, here $x^{a_1a_3}$ is obtained
from $x^{a_1a_2a_3}$ by deleting its $i$th position. We have that
$f(x^{00})=0, f(x^{01})=-1, f(x^{10})=1$,  which implies that $f$
is a cross and we are in the second case.

\end{proof}

\end{proof}

\noindent {\small {\bf Acknowledgments.} The reported study was
funded by RFBR according to the research project N 18-31-00126.}


\begin{thebibliography}{1}


\bibitem{AvgSol} S. V. Avgustinovich, F. I. Solov'eva, On the Nonsystematic
Perfect Binary Codes, Problems Inform. Transmission. 32(3) (1996),
258-261

\bibitem{ZinRifBorg} J. Borges, J. Rifa, V. A. Zinoviev, On Completely Regular
Codes. https://arxiv.org/abs/1703.08684

\bibitem{BRZlin} J. Borges, J. Rifa, V. A. Zinoviev, On q-ary linear completely
regular codes with r=2 and antipodal dual, In Adv. in Math. of
Comm. 4(4) (2010), 567-578.

\bibitem{Bra}M. Braun, T. Etzion, P.R.J. Ostergard, A. Vardy, A. Wassermann,
Existence of q-analogs of Steiner systems, Forum Math. Pi 4 (2016)
14. http://dx.doi.org/10.1017/fmp.2016.5


\bibitem{BCN} A.E. Brouwer, A.M. Cohen, and A. Neumaier, Distance-regular graphs, Ergebnisse der Mathematik und ihrer Grenzgebiete (3), Springer-Verlag, Berlin, 1989.

\bibitem{CDZ}  D.M. Cvetkovic, M. Doob, H. Sachs, Spectra of graphs, Academic
Press, New York - London, 1980.

\bibitem{CL} Cameron P.J., Liebler R.A.: Tactical decompositions and orbits of projective groups. Linear Algebra Appl.
46, (1982), 91-102.

\bibitem{Del73}
    P. Delsarte, An Algebraic Approach to the Association Schemes of Coding
    Theory,
     Philips Res. Rep. Suppl. 10 (1973), 1-97.


\bibitem{FDF1} D. G. Fon-Der-Flaass,  Perfect 2-colorings of a
hypercube. Siberian Mathematical Journal. 48(4) (2007), 740-745.


\bibitem{FDF2} D. G. Fon-Der-Flaass, A bound on correlation immunity.
Siberian Electronic Math. Reports. 4 (2007), 133-135.


\bibitem{FDF3} D.G. Fon-Der-Flaass, Perfect 2-coloring of the 12-cube
that attain the bound on correlation immunity, Siberian Electronic
Math. Reports. 4 (2007), 292-295.


\bibitem{KV} D. S. Krotov, K. V. Vorob'ev, On unbalanced
Boolean functions attaining the bound 2n/3-1 on the correlation
immunity https://arxiv.org/abs/1812.02166


\bibitem{KKM} D. Krotov, J. Koolen, W. Martin, sites.google.com/site/completelyregularcodes/h/r1




\bibitem{KLMT} J. H. Koolen, W. S. Lee, W. J. Martin, and H.
Tanaka, Arithmetic completely regular codes, Discrete Mathematics
and Theoretical Computer Science 17.3. (2016), 59-76.

\bibitem{Mey} A. Meyerowitz,
Cycle-balanced Conditions for Distance-regular Graphs, Discrete
Mathematics. 264(3) (2003), 149-166

\bibitem{Martin} W.J. Martin, Completely regular designs of strength one,
Journal of Algebraic Combinatorics.  3(2) (1994), 177-185.

\bibitem{Sol} F.I. Solov'eva, Class of closed-packed binary codes
generated by q-ary codes, Methody Discretnogo Analiza. 48 (1989),
 70-72 (In Russian).

\bibitem{V1} A. Valyuzhenich, Minimum supports of eigenfunctions of Hamming
graphs. Discrete Mathematics. 340(5) (2017), 1064-1068.

\bibitem{V2} A. Valyuzhenich, K. Vorob'ev,
Minimum supports of functions on the Hamming graphs with spectral
constraints, Discrete Mathematics. 342(5) (2019), 1351-1360.

\bibitem{Vor} K.Vorob'ev, Alphabet lifting construction of
equitable partitions of Hamming graphs, Abstracts of Graphs,
Groups, Graphs and Groups, Representations and Relations
Novosibirsk, Russia, August 06 - 19. (2018),  86.


\bibitem{Zin} V.A. Zinov'ev. Combinatorial methods for the
construction and analysis of nonlinear error-correcting codes,
Doc. D. Thesis, Moscow (1988) (In Russian).

\end{thebibliography}
\end{document}